\documentclass[12pt]{article}
 \usepackage{latexsym}
 \usepackage{amssymb}
 \usepackage{graphicx}
 \usepackage{amsmath}

 \newtheorem{Theorem}{Theorem}

\newtheorem{Corollary}{Corollary}

\newenvironment{proof}[1][Proof]{\begin{trivlist}
\item[\hskip \labelsep {\bfseries #1}]}{\end{trivlist}}

\newcommand{\A}{{\cal A}}
\newcommand{\B}{{\cal B}}
\newcommand{\C}{{\cal C}}
\newcommand{\E}{{\cal E}}
\newcommand{\I}{{\cal I}}

\newcommand{\M}{{\cal M}}

\newcommand{\uu}{{\bf u}}

\newcommand{\x}{{\bf x}}
\newcommand{\y}{{\bf y}}

\newcommand{\qed}{\nobreak \ifvmode \relax \else
      \ifdim\lastskip<1.5em \hskip-\lastskip
      \hskip1.5em plus0em minus0.5em \fi \nobreak
      \vrule height0.75em width0.5em depth0.25em\fi}

\def \ep{\hbox{ }\hfill$\Box$}

\begin{document}
\title{An Even Order Symmetric B Tensor is Positive Definite\footnote{To appear in: Linear Algebra and Its Applications.}}

\author{Liqun Qi
\thanks{Department of Applied Mathematics, The Hong Kong Polytechnic University, Hung Hom,
Kowloon, Hong Kong. Email: maqilq@polyu.edu.hk. This author's work
was supported by the Hong Kong Research Grant Council (Grant No.
PolyU 502510, 502111, 501212 and 501913).}, \quad    Yisheng Song
\thanks{ \footnotesize{School of Mathematics and Information
Science,
 Henan Normal University, XinXiang HeNan,  P.R. China, 453007.
Email: songyisheng1@gmail.com.} This author's work was supported by
the National Natural Science Foundation of P.R. China (Grant No.
11171094, 11271112).  His work was partially done when he was
visiting The Hong Kong Polytechnic University.} }

\date{\today} \maketitle

\begin{abstract}
\noindent  
It is easily checkable if a given tensor is a B tensor, or a B$_0$
tensor or not. In this paper, we show that a symmetric B tensor can
always be decomposed to the sum of a strictly diagonally dominated
symmetric M tensor and several positive multiples of partially all
one tensors, and a symmetric B$_0$ tensor can always be decomposed
to the sum of a diagonally dominated symmetric M tensor and several
positive multiples of partially all one tensors. When the order is
even, this implies that the corresponding B tensor is positive
definite, and the corresponding B$_0$ tensor is positive
semi-definite.    This gives a checkable sufficient condition for
positive definite and semi-definite tensors.   This approach is
different from the approach in the literature for proving a
symmetric B matrix is positive definite, as that matrix approach
cannot be extended to the tensor case.

\noindent {\bf Key words:}\hspace{2mm} Positive definiteness, B
tensor, B$_0$ tensor, M tensor, partially all one tensor.
\vspace{3mm}

\noindent {\bf AMS subject classifications (2010):}\hspace{2mm}
47H15, 47H12, 34B10, 47A52, 47J10, 47H09, 15A48, 47H07.
  \vspace{3mm}

\end{abstract}


\section{Introduction}
\hspace{4mm} Denote $[n] := \{ 1, \cdots, n \}$. A real $m$th order
$n$-dimensional tensor  $\A = (a_{i_1\cdots i_m})$ is a multi-array
of real entries $a_{i_1\cdots i_m}$, where $i_j \in [n]$ for $j \in
[m]$.  All the real $m$th order $n$-dimensional tensors form a
linear space of dimension $n^m$.   Denote this linear space by
$T_{m, n}$.   For $i \in [n]$, we call $a_{ii_2\cdots i_m}$ for $i_j
\in [n]$, $j = 2, \cdots, m$, the entries of $\A$ in the $i$th row,
where $a_{i\cdots i}$ is the $i$th diagonal entry of $\A$, while the
other entries are the off-diagonal entries of $\A$ in the $i$th row.

Let $\A = (a_{i_1\cdots i_m}) \in T_{m, n}$. If the entries
$a_{i_1\cdots i_m}$ are invariant under any permutation of their
indices, then $\A$ is called a {\bf symmetric tensor}.  All the real
$m$th order $n$-dimensional symmetric tensors form a linear subspace
of $T_{m, n}$.  Denote this linear subspace by $S_{m, n}$. Let $\A =
(a_{i_1\cdots i_m}) \in T_{m, n}$ and $\x \in \Re^n$.   Then $\A
x^m$ is a homogeneous polynomial of degree $m$, defined by
$$\A x^m = \sum_{i_1,\cdots, i_m=1}^n a_{i_1\cdots i_m}x_{i_1}\cdots
x_{i_m}.$$   A tensor $\A \in T_{m, n}$ is called {\bf positive
semi-definite} if for any vector $\x \in \Re^n$, $\A \x^m \ge 0$,
and is called {\bf positive definite} if for any nonzero vector $\x
\in \Re^n$, $\A \x^m > 0$.   Clearly, if $m$ is odd, there is no
nonzero positive semi-definite tensors.   Positive definiteness and
semi-definiteness of real symmetric tensors and their corresponding
homogeneous polynomials have applications in automatical control
\cite{BM, HH, JM, WQ}, polynomial problems \cite{Re, Sh}, magnetic
resonance imaging \cite{CDHS, HHNQ, QYW, QYX} and spectral
hypergraph theory \cite{HQ, HQ14, HQS, HQX, LQY, Qi14, QSW}. In
\cite{Qi}, Qi introduced H-eigenvalues and Z-eigenvalues for real
symmetric tensors, and showed that an even order real symmetric
tensor is positive (semi-)definite if and only if all of its
H-eigenvalues, or all of its Z-eigenvalues, are positive
(nonnegative).   In matrix theory, it is well-known that a strictly
diagonally dominated symmetric matrix is positive definite and a
diagonally dominated symmetric matrix is positive semi-definite.
Here, we may also easily show that an even order strictly diagonally
dominated symmetric tensor is positive definite and an even order
diagonally dominated symmetric tensor is positive semi-definite.  We
will show this in Section 2.   Based upon this, we know that the
Laplacian tensor in spectral hypergraph theory is positive
semi-definite \cite{HQ14, HQS, HQX, Qi14, QYX}. Song and Qi
\cite{SQ} showed that an even order Hilbert tensor is positive
definite.   This also extends the matrix result that a Hilbert
matrix is positive definite.     In matrix theory, a completely
positive tensor is positive semi-definite, and a diagonally
dominated symmetric nonnegative tensor is completely positive. In
\cite{QXX}, completely positive tensors were introduced. An even
order completely positive tensor is also positive semi-definite.
Then, it was shown in \cite{QXX} that a strongly symmetric,
hierarchically dominated nonnegative tensor is completely positive.
These are some checkable sufficient conditions for positive definite
or semi-definite tensors in the literature.

In the matrix literature, there is another easily checkable
sufficient condition for positive definite matrices.   It is easy to
check a given matrix is a B matrix or not \cite{Pe, Pe1}.   A B
matrix is a P matrix \cite{Pe}. It is well-known that a symmetric
matrix is a P matrix if and only it is positive definite \cite[Pages
147, 153]{CPS}. Thus, a symmetric B matrix is positive definite.

P matrices and B matrices were extended to P tensors and B tensors
in \cite{SQ1}.  It is easy to check a given tensor is a B tensor or
not, while it is not easy to check a given tensor is a P tensor or
not.   It was proved there that a symmetric tensor is a P tensor if
and only it is positive definite.    However, it was not proved in
\cite{SQ1} if an even order B tensor is a P tensor or not, or if an
even order symmetric B tensor is positive definite or not.  As
pointed out in \cite{SQ1}, an odd order identity tensor is a B
tensor, but not a P tensor.  Thus we know that an odd order B tensor
may not be a P tensor.

The B tensor condition is not so strict compare with the strongly
diagonal dominated tensor condition if the tensor is not sparse. A
tensor in $T_{m, n}$ is strictly diagonally dominated tensor if
every diagonal entry of that tensor is greater than the sum of the
absolute values of all the off-diagonal entries in the same row. For
each row, there are $n^{m-1} - 1$ such off-diagonal entries. Thus,
this condition is quite strict when $n$ and $m$ are big and the
tensor is not sparse. A tensor in $T_{m, n}$ is a B tensor if for
every row of the tensor, the sum of all the entries in that row is
positive, and each off-diagonal entry is less than the average value
of the entries in the same row. An initial numerical experiment
indicated that for $m =4$ and $n=2$, a symmetric B tensor is
positive definite. Thus, it is possible that an even order symmetric
B tensor is positive definite.   If this is true, we will have an
easily checkable, not very strict, sufficient condition for positive
definite tensors.

However, the technique in \cite{Pe} cannot be extended to the tensor
case. It was proved in \cite{Pe} that the determinant of every
principal submatrix of a B matrix is positive.    Thus, a B matrix
is a P matrix.   It was pointed out in \cite{Qi} that the
determinant of every principal sub-tensor of a symmetric positive
definite tensor is positive, but this is only a necessary, not a
sufficient condition for symmetric positive definite tensors. Hence,
the technique in \cite{Pe} cannot be extended to the tensor case.

In \cite{SQ1}, P tensors were defined by extending an alternative
definition for P matrices.   But it is still unknown if an even
order B tensor is a P tensor or not.

In this paper, we use a new technique to prove that an even order
symmetric B tensor is positive definite.   We show that a symmetric
B tensor can always be decomposed to the sum of a strictly
diagonally dominated symmetric M tensor and several positive
multiples of partially all one tensors, and a symmetric B$_0$ tensor
can always be decomposed to the sum of a diagonally dominated
symmetric M tensor and several positive multiples of partially all
one tensors. Even order partially all one tensors are positive
semi-definite. As stated before, an even order diagonally dominated
symmetric tensor is positive semi-definite, and an even order
strictly diagonally dominated symmetric tensor is positive definite.
Therefore, when the order is even, these imply that the
corresponding symmetric B tensor is positive definite, and the
corresponding symmetric B$_0$ tensor is positive semi-definite.
Hence, this gives an easily checkable, not very strict, sufficient
condition for positive definite and semi-definite tensors.

In the next section, we study diagonally dominated symmetric
tensors.  In Section 3, we define B, B$_0$ and partially all one
tensors, and discuss their general properties. The main result is
given in Section 4. We make some final remarks and raise some
further questions in Section 5.

Throughout this paper, we assume that $m \ge 2$ and $n \ge 1$.   We
use small letters $x, u, v, \alpha, \cdots$, for scalers, small bold
letters $\x, \y, \uu, \cdots$, for vectors, capital letters $A, B,
\cdots$, for matrices, calligraphic letters $\A, \B, \cdots$, for
tensors. All the tensors discussed in this paper are real.

\section{Diagonally Dominated Symmetric Tensors}
\hspace{4mm}

We define the generalized Kronecker symbol as
$$\delta_{i_1\cdots \i_m} = \begin{cases}1, {\rm if}\ i_1 = \cdots = i_m, \\ 0, {\rm otherwise}. \end{cases}$$
The tensor $\I = (\delta_{i_1\cdots \i_m})$ is called the {\bf
identity tensor} of $T_{m, n}$.

Let $\A = (a_{i_1\cdots i_m}) \in T_{m, n}$.   If for $i \in [n]$,
$$a_{i\cdots i} \ge \sum \{ |a_{ii_2\cdots i_m}| : i_j \in [n], j = 2,
\cdots, m, \delta_{ii_2\cdots i_m} = 0 \},$$ then $\A$ is called a
{\bf diagonally dominated tensor}.    If for $i \in [n]$,
$$a_{i\cdots i} > \sum \{ |a_{ii_2\cdots i_m}| : i_j \in [n], j = 2,
\cdots, m, \delta_{ii_2\cdots i_m} = 0 \},$$ then $\A$ is called a
{\bf strictly diagonally dominated tensor}.

Let $\A = (a_{i_1\cdots i_m}) \in T_{m, n}$ and $\x \in {\boldmath
C}^n$.   Define $\A \x^{m-1}$ as a vector in ${\boldmath C}^n$ with
its $i$th component as
$$\left(\A \x^{m-1}\right)_i = \sum_{i_2, \cdots, i_m=1}^n a_{ii_2\cdots
i_m}x_{i_2}\cdots x_{i_m}$$ for $i \in [n]$. For any vector $\x \in
{\boldmath C}^n$, define $\x^{[m-1]}$ as a vector in ${\boldmath
C}^n$ with its $i$th component defined as $x_i^{m-1}$ for $i \in
[n]$. Let $\A \in T_{m, n}$.  If there is a nonzero vector $\x \in
{\boldmath C}^n$ and a number $\lambda \in {\boldmath C}$ such that
\begin{equation} \label{eig}
\A \x^{m-1} = \lambda \x^{[m-1]},
\end{equation}
then $\lambda$ is called an {\bf eigenvalue} of $\A$ and $\x$ is
called an {\bf eigenvector} of $\A$, associated with $\lambda$.   If
the eigenvector $\x$ is real, then the eigenvector $\lambda$ is also
real.  In this case, $\lambda$ and $\x$ are called an {\bf
H-eigenvalue} and an {\bf H-eigenvector} of $\A$, respectively. The
maximum modulus of the eigenvalues of $\A$ is called the {\bf
spectral radius} of $\A$, and denoted as $\rho(\A)$.   Eigenvalues
and H-eigenvalues were first introduced in \cite{Qi} for symmetric
tensors.    The following theorem is from \cite[Theorem 5]{Qi}.

\begin{Theorem} \label{t1}
Suppose that $\A \in S_{m, n}$ and $m$ is even.  Then $\A$ always
has H-eigenvalues.    $\A$ is positive semi-definite if and only if
all of its H-eigenvalues are nonnegative.   $\A$ is positive
definite if and only if all of its H-eigenvalues are positive.
\end{Theorem}

The following theorem is from \cite[Theorem 6]{Qi}.   Theorem 6 of
\cite{Qi} is restricted to symmetric tensors.   But it is true for
nonsymmetric tensors, and the proof is the same.

\begin{Theorem} \label{t2}
Suppose that $\A \in T_{m, n}$.  Then the eigenvalues $\lambda$ of
$\A$ satisfy the following constraints: for $i \in [n]$,
$$|\lambda - a_{i\cdots i}| \le \sum \{ |a_{ii_2\cdots i_m}| : i_j \in [n], j = 2,
\cdots, m, \delta_{ii_2\cdots i_m} = 0 \}.$$
\end{Theorem}

We now have the following theorem.

\begin{Theorem}\label{t3}
Let $\A \in S_{m, n}$ and $m$ be even.   If $\A$ is diagonally
dominated, then $\A$ is positive semi-definite.  If $\A$ is strictly
diagonally dominated, then $\A$ is positive definite.
 \end{Theorem}
\begin{proof}  By Theorem \ref{t2} and the definition of diagonally dominated and strictly diagonally dominated tensors, all the
H-eigenvalues of a diagonally dominated tensor, if exist, are
nonnegative, and all the H-eigenvalues of a strictly diagonally
dominated tensor, if exist, are positive.  The conclusions follow
from Theorem \ref{t1} now.  \ep
\end{proof}

\bigskip

Let $\A \in T_{m, n}$.   If all of the off-diagonal entries of $\A$
are non-positive, then $\A$ is called a {\bf Z tensor}.   If a Z
tensor $\A$ can be written as $\A = c\I - \B$, such that $\B$ is a
nonnegative tensor and $c \ge \rho(\B)$, then $\A$ is called an {\bf
M tensor} \cite{ZQZ}.   If $c > \rho(\B)$, then $\A$ is called a
{\bf strong M tensor} \cite{ZQZ}.  It was proved in \cite{ZQZ} that
a diagonally dominated Z tensor is an M tensor, and a strictly
diagonally dominated Z tensor is a strong M tensor. The properties
of M and strong M tensors may be found in \cite{DQW, HeH, ZQZ}.

\section{B, B$_0$ and Partially All One Tensors}
\hspace{4mm}   Let $\B = (b_{i_1\cdots i_m}) \in T_{m, n}$.  We say
that $\B$ is a {\bf B tensor} if for all $i \in [n]$
$$\sum_{i_2,\cdots, i_m=1}^{n}b_{ii_2i_3\cdots i_m}>0$$ and $$\frac1{n^{m-1}}\left(\sum_{i_2,\cdots, i_m=1}^{n}b_{ii_2i_3\cdots i_m}\right)>b_{ij_2j_3\cdots j_m} \mbox{ for all } (j_2, j_3, \cdots, j_m)\ne (i, i, \cdots, i).$$
We say that $\B$ is a {\bf B$_0$ tensor} if for all $i \in [n]$
$$\sum_{i_2,\cdots, i_m=1}^{n}b_{ii_2i_3\cdots i_m}\ge0$$ and $$\frac1{n^{m-1}}\left(\sum_{i_2,\cdots, i_m=1}^{n}b_{ii_2i_3\cdots i_m}\right)\ge b_{ij_2j_3\cdots j_m} \mbox{ for all } (j_2, j_3, \cdots, j_m)\ne (i, i, \cdots, i).$$

This definition is a natural extension of the definition of B
matrices \cite{Pe, Pe1, SQ1}.   It is easily checkable if a given
tensor in $T_{m, n}$ is a B tensor, or a $B_0$ tensor or not. As
discussed in the introduction, the definitions of B and B$_0$
tensors are not so strict, compared with the definitions of
diagonally dominated and strictly diagonally dominated tensors, if
the tensor is not sparse. We also can see that a Z tensor is
diagonally dominated if and only if it is a B$_0$ tensor, and  a Z
tensor is strictly diagonally dominated if and only if it is a B
tensor \cite{SQ1}.

 A tensor $\C \in T_{m, r}$  is called {\bf a principal sub-tensor}  of a tensor $\A = (a_{i_1\cdots i_m}) \in T_{m, n}$ ($1 \le r\leq n$) if there is a set $J$ that composed of $r$ elements in $[n]$ such that
 $$\C = (a_{i_1\cdots i_m}),\mbox{ for all } i_1, i_2, \cdots, i_m\in J.$$ This concept was first introduced and used in \cite{Qi} for symmetric tensor. We denote by $\A^J_r$ the principal sub-tensor of a tensor $\A \in T_{m, n}$ such that the entries of
 $\A^J_r$ are indexed by $J \subset [n]$ with $|J|=r$ ($1 \le r\leq
 n$).

 It was proved in \cite{SQ1} that all the principal sub-tensors of a
 B$_0$ tensor are B$_0$ tensors, and all the principal sub-tensors of a
 B tensor are B tensors.

 \bigskip
Suppose that $\A \in S_{m, n}$ has a principal sub-tensor $\A^J_r$
with $J \subset [n]$ with $|J|=r$ ($1 \le r\leq
 n$) such that all the entries of $\A^J_r$ are one, and all the
 other entries of $\A$ are zero.   Then $\A$ is called a {\bf partially
 all one tensor}, and denoted by $\E^J$.  If $J = [n]$, then we denote $\E^J$ simply by $\E$ and call
 it an {\bf all one tensor}.   An even order partially all
 one tensor is positive semi-definite.  In fact, when $m$ is even, if we
 denote by $\x_J$ the $r$-dimensional sub-vector of a vector $\x \in \Re^n$,
 with the components of $\x_J$ indexed by $J$, then for any $\x \in
 \Re^n$, we have
 $$\E^J\x^m = \left( \sum \{ x_j : j \in J \} \right)^m \ge 0.$$

\section{Decomposition of B Tensors} \hspace{4mm}

We now prove the main result of this paper.

\begin{Theorem} \label{t4}
Suppose that $\B = (b_{i_1\cdots i_m}) \in S_{m, n}$ is a symmetric
B$_0$ tensor.   Then either $\B$ is a diagonally dominated symmetric
M tensor itself, or we have
\begin{equation} \label {e1}
\B = \M + \sum_{k=1}^s h_k \E^{J_k},
\end{equation}
where $\M$ is a diagonally dominated symmetric M tensor, $s$ is a
positive integer, $h_k
> 0$ and $J_k \subset [n]$,  for $k = 1, \cdots , s$,
and $J_k \cap J_l = \emptyset$, for $k \not = l, k$ and $l = 1,
\cdots , s$ when $s > 1$.   If furthermore $\B$ is a B tensor, then
either $\B$ is a strictly diagonally dominated symmetric M tensor
itself, or we have (\ref{e1}) with $\M$ as a strictly diagonally
dominated symmetric M tensor.
   An even order symmetric B$_0$ tensor is positive semi-definite. An even order
symmetric B tensor is positive definite.
\end{Theorem}
\begin{proof}    We now prove the first conclusion.  Suppose that $\B = (b_{i_1\cdots i_m}) \in S_{m, n}$ is a symmetric B$_0$
tensor.    Define $\hat J(\B) \subset [n]$ as
$$\hat J(\B) = \{ i \in [n]: {\rm there\ is\ at\ least\ one\ positive\ }
{\rm off-diagonal}\ {\rm entry\ in\ the\ }i{\rm th\ row\ of\ }\B
\}.$$ If $\hat J(\B)$ is an empty set, then $\B$ is a Z tensor, thus
a diagonally dominated symmetric M tensor.   The conclusion holds in
this case.   Assume that $\hat J(\B)$ is not empty.   Let $\B_1 =
\B$.  For each $i \in \hat J(\B)$, let $d_i$ be the value of the
largest off-diagonal entry in the $i$th row of $\B_1$.   Let
$$J_1 = \hat J(\B_1).$$
We see that $J_1 \not = \emptyset$.   Let
$$h_1 = \min \{ d_i : i \in J_1 \}.$$
Then $h_1 > 0$.

Now consider $\B_2 = \B_1 - h_1\E^{J_1}$. It is not difficult to see
that $\B_2$ is still a symmetric B$_0$ tensor.

We now replace $\B_1$ by $\B_2$, and repeat this process. We see
that
$$\hat J(\B_2) = \{ i \in [n]: {\rm there\ is\ at\ least\ one\ positive\ }
{\rm off-diagonal}\ {\rm entry\ in\ the\ }i{\rm th\ row\ of\ }\B_2
\}$$ is a proper subset of $\hat J(\B_1)$. Repeat this process until
$\hat J(B_{s+1}) = \emptyset$.  Let $\M = B_{s+1}$.   We see that
(\ref{e1}) holds.    Then we have
$$\hat J(\B_{k+1}) = \hat J(\B_k) \setminus J_k,$$   for $k \in [s]$.  Thus,
$J_k \cap J_l = \emptyset$, for $k \not = l, k$ and $l = 1, \cdots ,
s$ when $s > 1$.   This proves the first conclusion.

Similarly, we may prove the second conclusion, i.e., if $\B$ is a B
tensor, then either $\B$ itself is a strictly diagonally dominated
symmetric M tensor, or in (\ref{e1}), $\M$ is a strictly diagonally
dominated symmetric M tensor.

Suppose now $\B$ is a symmetric B$_0$ tensor and $m$ is even.  If
$\B$ itself is a diagonally dominated symmetric M tensor, then it is
positive semi-definite by Theorem \ref{t3}.   Otherwise, (\ref{e1})
holds with $s
> 0$. Let $\x \in \Re^n$. Then by (\ref{e1}),
$$\B \x^m = \M \x^m + \sum_{k=1}^s h_k \E^{J_k} \x^m = \M \x^m + \sum_{k=1}^s h_k \|\x_{J_k}\|_m^m \ge \M \x^m \ge 0,$$
as by Theorem \ref{t3}, a diagonally dominated symmetric M tensor is
positive semi-definite.  This proves the third conclusion.

The fourth conclusion can be proved similarly.
 \ep \end{proof}

For nonsymmetric B and B$_0$ tensors, some decomposition results may
still be obtained.    However, in this case, we cannot establish
positive definiteness or semi-definiteness results as Theorems
\ref{t1} and \ref{t3} cannot be applied to nonsymmetric tensors.

By this theorem and Theorem \ref{t1}, we have the following
 corollary.

 \begin{Corollary} \label{c1}
 All the H-eigenvalues of an even order symmetric B$_0$ tensor are
 nonnegative.   All the H-eigenvalues of an even order symmetric B tensor are
 positive.
 \end{Corollary}

\section{Final Remarks and Further Questions}
\hspace{4mm}   Theorem \ref{t4} gives an easily checkable sufficient
condition for positive definite and semi-definite tensors.   It is
much more general compared with Theorem \ref{t3}.   The proof
technique of Theorem \ref{t4} is totally different that in the B
matrix literature \cite{Pe, Pe1}.   It decomposes a symmetric B
tensor as the sum of two kinds of somewhat basic tensors: strictly
diagonally dominated symmetric M tensors and positive multiples of
partially all one tensors.

\medskip

{\bf Question 1} Can we apply this technique to give more general
sufficient conditions for positive definite and semi-definite
tensors?

\medskip

In \cite{SQ1}, it was proved that an even order symmetric tensor is
positive definite if and only if it is a P tensor, and an even order
symmetric tensor is positive semi-definite if and only if it is a
P$_0$ tensor.   Thus, an even order symmetric B tensor is a P tensor
and an even order symmetric B$_0$ tensor is a P$_0$ tensor.

\medskip

{\bf Question 2}  Can we show that an even order non-symmetric B
tensor is a P tensor and an even order non-symmetric B$_0$ tensor is
a P$_0$ tensor?    After the early draft of this paper at arXiv,
Yuan and You \cite{YY} gave a counter example to answer this
question.

\medskip

In the literature, we know that several classes of tensors have the
following two properties:

a). If the order is even, then they are positive semi-definite;

b). If the order is odd, then their H-eigenvalues, if exist, are
nonnegative.

This includes diagonally dominated tensors discussed in Section 2 of
this paper, complete Hankel tensors and strong Hankel tensors
\cite{Qi14a}, completely positive tensors \cite{QXX} and P$_0$
tensors \cite{SQ1}.   Some of them guarantee that H-eigenvalues
exist even when the order is odd.

\medskip

{\bf Question 3}  Does an odd order symmetric B$_0$ tensor always
have H-eigenvalues?   If such H-eigenvalues exist, are they always
nonnegative?

\bigskip

{\bf Acknowledgment} We are thankful to Pingzhi Yuan, Lihua You,
Zhongming Chen and the referee for their comments.


\end{document}